\documentclass[letterpaper, 12pt]{amsart}

\usepackage{latexsym,mathrsfs,amscd,amsmath,verbatim,amssymb,calc,enumerate,amsxtra,color,ifpdf}
\usepackage[all]{xy}
\ifpdf%
  \usepackage{hyperref}
\else%
  \usepackage[hypertex]{hyperref}
\fi%

\allowdisplaybreaks[1]

\theoremstyle{plain}
\newtheorem{theorem}{Theorem}[section]
\newtheorem{corollary}[theorem]{Corollary}

\newtheorem{lemma}[theorem]{Lemma}

\newtheorem*{theorem*}{Theorem}

\theoremstyle{remark}

\theoremstyle{definition}
\newtheorem{definition}[theorem]{Definition}

\numberwithin{equation}{theorem}

{\noindent\ignorespaces}%
{\par\noindent%
\ignorespacesafterend}

\newcommand{\ZZ}{\mathbb{Z}}
\newcommand{\m}{\mathfrak{m}}

\newcommand{\p}{\mathfrak{p}}
\newcommand{\q}{\mathfrak{q}}

\newcommand{\Ass}{\operatorname{Ass}}
\newcommand{\Spec}{\operatorname{Spec}}

\newcommand{\Max}{\operatorname{Max}}
\newcommand{\Supp}{\operatorname{Supp}}
\newcommand{\Ann}{\operatorname{Ann}}

\newcommand{\Hom}{\operatorname{Hom}}
\newcommand{\inj}{\operatorname{inj}}

\newcommand{\rank}{\operatorname{rank}}

\newcommand{\cog}{\operatorname{cog}}
\newcommand{\Soc}{\operatorname{Soc}}

\renewcommand{\ge}{\geqslant} \renewcommand{\le}{\leqslant}

\begin{document}

\title[Injective capacity and cogeneration]{Injective capacity and cogeneration}
\author{Robin Baidya and Yongwei Yao}
\email{rbaidya@utk.edu}
\address{Department of Mathematics, The University of Tennessee, Knoxville, Tennessee 37996}
\email{yyao@gsu.edu}
\address{Department of Mathematics and Statistics, Georgia State University, Atlanta, Georgia 30303}

\date{\today}
\begin{abstract}
Let $M$ and $N$ be modules over a commutative ring~$R$ with $N$ Noetherian.  We define the \textit{injective capacity of $M$ with respect to $N$ over~$R$} to be the supremum of the values $t$ for which $N^{\oplus t}$ embeds into~$M$.  In a dual fashion, we deem the \textit{number of cogenerators of $N$ with respect to $M$ over~$R$} to be the infimum of the numbers $t$ for which $N$ embeds into~$M^{\oplus t}$.  We demonstrate that the global injective capacity is the infimum of its local analogues and that the global number of cogenerators is the supremum of the corresponding local invariants.  We also prove enhanced versions of these statements and consider the graded case.
\end{abstract}
\commby{}
\maketitle

\section{Introduction}\label{sec:intro}

Our goal in this paper is to prove various refinements of the following assertion:

\begin{theorem}[Main Theorem]\label{theorem:intro}
If $M$ and $N$ are modules over a commutative ring $R$ with $N$ Noetherian, then the existence of an injection in $\Hom_R(N,M)$ is a local property. 
\end{theorem}

There are two reasons why we find this result worthy of report.  One is the spectacular failure of the analogous statement for surjections, and the other is the voluminous response to this failure that has, at least to our knowledge, excluded consideration of the problem that we present here.  We open with a survey of the related literature on epimorphisms and proceed to illustrate that, inasmuch as a comparison can be made, the monic case affords stronger conclusions with weaker hypotheses.

Theorems on factor modules relevant to our discussion come in two types.  Findings of the first type deliver values $t$ for which $\Hom_R(M,N^{\oplus t})$ harbors a surjection, where $M$ and $N$ are modules over a ring~$R$.  For example, in~\cite{Bai1}, the first author shows that, if $R$ is a $d$-dimensional commutative Noetherian ring and $M$ and $N$ are finitely generated $R$-modules such that $N^{\oplus (t+d)}$ is locally a factor of $M$, then $N^{\oplus t}$ is globally a factor of~$M$~\cite[Theorem~1.1(1)]{Bai1}.  The same statement holds if we replace \textit{factor} with \textit{direct summand}~\cite[Theorem~1.1(2)]{Bai1}.  Specializations of these results are due to several authors:  Serre~\cite[\textit{Th\'eor\`eme~1}]{Ser} addresses the setting in which $M$ is projective and $N=R$; Bass~\cite[Theorem~8.2]{Bas} removes Serre's projective condition assuming that $t=1$; and De~Stefani--Polstra--Yao~\cite[Theorem~3.13]{DSPY} extends Bass by allowing $t$ to be an arbitrary positive integer.  Additional variations on Serre's Splitting Theorem, all of which involve some notion of dimension, appear in Coquand--Lombardi--Quitt\'e~\cite[Corollary~3.2]{CLQ}, Heitmann~\cite[Corollary~2.6]{Heit}, and Stafford~\cite[Proposition~5.5]{Sta}.  These results are sharp to the extent that the coordinate ring of the real $d$-sphere admits an indecomposable rank-$d$ projective module for every positive even integer~$d$~\cite[Theorem~3]{Swa2}.  So, the existence of an epimorphism between two fixed modules, though not a local property, does follow from, and generally requires, certain dimensional restrictions, given that some finiteness conditions are in place.

This fact naturally affects module generation, the subject of the second type of result that concerns us here.  For modules $M$ and $N$ over a ring~$R$, we say that \textit{$u$ copies of $M$ generate $N$} if $\Hom_R(M^{\oplus u},N)$ contains an onto map.  The local existence of such integers $u$ does not ensure global existence:  If $R=M$ is a direct product of infinitely many fields and $N$ is the direct sum of the local rings of~$R$, then $N$ is finitely generated locally but not globally.  In order to avoid this sort of situation, theorists traditionally assume global existence when applying local data to the issue of module generation.  Exploiting this concession, the first author certifies in a forthcoming paper that, if $M$ and $N$ are finitely generated modules over a $d$-dimensional commutative Noetherian ring $R$ such that $t-d$ copies of $M$ generate~$N$ locally, then $t$ copies of $M$ generate~$N$ globally~\cite{Bai4}.  In the case that $M=R$, the preceding assertion is a theorem of Forster~\cite[{\textit{Satz~2}}]{For}.  A refinement of Forster achieved by Swan~\cite[Theorem~2]{Swa} is the celebrated theorem now bearing the names of both men.  Further elaborations on Forster's Theorem are due to Coutinho~\cite[Corollary~8.5]{Cou2}~\cite[Theorem~5.4]{Cou}, Eisenbud--Evans~\cite[Theorem~B]{EE}, Lyubeznik~\cite[Theorem~2]{Lyu}, and Warfield~\cite[Theorem~2]{War2}.  Despite the many ways that Forster's Theorem has been extended, Forster's local condition cannot be improved in general:  For all nonnegative integers $d$ and $t$ with $d<t$, there is a $d$-dimensional commutative ring $R$ admitting a projective module of rank $t-d$ that is not a factor of~$R^{\oplus (t-1)}$~\cite[Theorem~4]{Swa2}.  Thus, once again, we are confronted with a question on epimorphisms that resists reduction to the local realm and exhibits strong ties to dimension, even with the provision of several finiteness hypotheses.

In contrast, our solution to the problem of embeddability boils down to the local case and obviates dimensional considerations.  Moreover, the Noetherian condition on the source of our maps, which constitutes our only finiteness assumption, cannot be removed:  Adapting a prior example, we observe that, if $R=N$ is a direct product of infinitely many fields and $M$ is the direct sum of the local rings of~$R$, then $N$ embeds into~$M$ locally but not globally.  Assuming, on the other hand, that $M$ and $N$ are modules over a commutative ring $R$ with $N$ Noetherian, we may marshal the fact (Lemma~\ref{lemma:reduction} below) that a map $h\in\Hom_R(N,M)$ is monic if and only if $h_{\p}$, when restricted to the socle of~$N_{\p}$, is monic for every $\p$ in the finite set $\Ass_R(N)$.  This characterization of injectivity allows us to use a ``general position" argument in the style of~\cite[Sections~6 and~7]{Bai1}, which would suffice to prove our main theorem.  To complement the existing literature on direct summands of factor modules and on factor modules of direct sums, we could then replace $\Hom_R(N,M)$ in our main theorem with $\Hom_R(N^{\oplus t},M)$ or $\Hom_R(N,M^{\oplus t})$ for every nonnegative integer~$t$.

Instead, we take the following route:  We begin by laying down the technical foundations for our argument in Section~\ref{sec:tech}, introducing two invariants of the triple $(M,N,R)$:  the \textit{injective capacity of $M$ with respect to $N$ over~$R$} and the \textit{number of cogenerators of $N$ with respect to $M$ over~$R$}.  In Section~\ref{sec:root-rank}, we establish results on roots of polynomials and ranks of matrices to which we appeal later in the paper.  Section~\ref{sec:main} offers enhancements of our main theorem using injective capacities and numbers of cogenerators.  Our main theorem follows as a corollary of these statements.  Finally, in Section~\ref{sec:inj-graded}, we initiate an investigation on embeddings of graded modules.

\section{Conventions and definitions}\label{sec:tech}

In this section, we flesh out our conventions for the rest of the paper and cover definitions that will streamline our discussion.

Throughout this paper, the letter $R$ refers to a commutative ring with unity.  The set of all maximal ideals of~$R$, called the \textit{maximal spectrum of~$R$}, is denoted $\Max(R)$.  The set of all prime ideals of~$R$, referred to as the \textit{prime spectrum of~$R$}, is written as $\Spec(R)$.  For every $\p\in\Spec(R)$, the symbol $\kappa(\p)$ represents the \textit{residue field of $R$ at $\p$} or, in other words, the ring $R_{\p}/\p_{\p}$.  

Every $R$-module in this paper is standard.  The letter $M$ refers to an arbitrary $R$-module, and $N$ stands for a Noetherian $R$-module.   The \textit{annihilator of $N$ in $R$}, written as $\Ann_R(N)$, is the largest ideal $I$ of $R$ satisfying $IN=0$.  We signify with $\Supp_R(N)$ the set of all $\p\in\Spec(R)$ such that $N_{\p}\neq 0$, and we call this set the \textit{support of $N$ over $R$}.  For every $\p\in\Spec(R)$, the \textit{socle of $N_{\p}$ over $R_{\p}$}, denoted $\Soc_{R_{\p}}(N_{\p})$, is the $\kappa(\p)$-module consisting of all elements of $N_{\p}$ annihilated by~$\p_{\p}$.  An \textit{associated prime of $N$ in $R$} is a prime $\p$ of $R$ such that $\Soc_{R_{\p}}(N_{\p})\neq 0$.  The set $\Ass_R(N)$ is the finite set of all associated primes of $N$ in $R$.  For every $R$-module~$L$ and nonnegative integer~$t$, the notation $L^{\oplus t}$ serves as shorthand for the direct sum of $t$ copies of~$L$, with $L^{\oplus 0}$ signifying the zero $R$-module.  $\Hom_R(N,M)$ refers to the $R$-module of all $R$-linear maps from $N$ to~$M$.  For every nonnegative integer~$t$, we view a member of $\Hom_R(N^{\oplus t},M)$ as a row $(h_1,\ldots,h_t)$, and we view a member of $\Hom_R(N,M^{\oplus t})$ as a column
\[
\begin{pmatrix}
h_1 \\
\vdots \\
h_t \\
\end{pmatrix},
\]
where $h_1,\ldots,h_t\in\Hom_R(N,M)$.  Often, we write $(h_1,\ldots,h_t)^{\top}$ to denote the transpose of a row.  For an $R$-submodule $F$ of $\Hom_R(N,M)$, context will determine the meaning of the symbol $F^{\oplus t}$:  If we write $F^{\oplus t}\subseteq\Hom_R(N^{\oplus t},M)$, then $F^{\oplus t}$ refers to the $R$-module of all rows $(f_1,\ldots,f_t)$, where $f_1,\ldots,f_t\in F$; if we write $F^{\oplus t}\subseteq\Hom_R(N,M^{\oplus t})$, then $F^{\oplus t}$ represents the $R$-module of all columns of length $t$ with entries in $F$.

We now expand on two terms first mentioned in our introduction:

\begin{definition}\label{definition:inj}
	 We let $\inj_R^F(M,N)$ denote the supremum of the nonnegative integers $t$ for which an injection exists in $F^{\oplus t}\subseteq\Hom_R(N^{\oplus t},M)$.  We call $\inj^F_R(M,N)$ the \textit{global injective capacity of $M$ with respect to $N$ over~$R$ when restricted to $F$}.  For every $\p\in\Spec(R)$, we call $\inj^{F_{\p}}_{R_{\p}}(M_{\p},N_{\p})$ the \textit{local injective capacity of $M$ with respect to $N$ over $R$ when restricted to $F$ at~$\p$}.  We omit the superscript $F$ and the phrase \textit{when restricted to $F$} in the case that $F=\Hom_R(N,M)$.  We set $\sup(\varnothing)=0$ and $\sup(\{0,1,2,\ldots\})=\infty$.
\end{definition}

\begin{definition}\label{definition:cog}
	The \textit{global number of cogenerators of $N$ with respect to $M$ over~$R$ when restricted to $F$}, denoted $\cog_R^F(N,M)$, refers to the infimum of the nonnegative integers $t$ such that there is an injection in $F^{\oplus t}\subseteq\Hom_R(N,M^{\oplus t})$.  For every $\p\in\Spec(R)$, the \textit{local number of cogenerators of $N$ with respect to $M$ over $R$ when restricted to $F$ at~$\p$} is the number $\cog^{F_{\p}}_{R_{\p}}(N_{\p},M_{\p})$.  If $F=\Hom_R(N,M)$, we withhold the superscript $F$ and all references to restriction.   We set $\inf(\varnothing)=\infty$ and $\inf(\{0,1,2,\ldots\})=0$.
\end{definition}

To recapitulate our main conventions, $R$ denotes a commutative ring; $M$ refers to an arbitrary $R$-module; and $N$ signifies a Noetherian $R$-module.

\section{Roots and ranks}\label{sec:root-rank}

In preparation for our ``general position" arguments, we prove three lemmas below that may attract independent interest.  Lemma~\ref{lemma:multi-poly-roots} describes a way to ensure, for example, that a Cartesian product of finite subsets of a field is not contained in (a fixed embedding of) a given affine algebraic set over that field.  Building on this result, Lemma~\ref{lemma:multi-rank-block} specifies an instance in which two matrices, one of which has large rank, can be combined to produce a matrix that still has large rank.  Lemma~\ref{lemma:multi-rank} is a variant of Lemma~\ref{lemma:multi-rank-block} that we use in the last section of our paper when discussing graded modules.

\begin{lemma}\label{lemma:multi-poly-roots}
Let $k_1,\ldots,k_s$ be fields and $x_1,\ldots,x_t$ indeterminates.  For every $i\in\{1,\ldots,s\}$, let $f_i$ be a nonzero polynomial in $k_i[x_1,\dotsc, x_t]$ such that every monomial in its support divides~$x_1^{n_{i1}}\dotsb x_t^{n_{it}}$.  For every $j\in\{1,\ldots,t\}$, let $C_j$ be a set with $|C_j| > \sum_{i=1}^s n_{ij}$, and suppose that there are embeddings $\phi_{j1},\ldots,\phi_{js}$  of $C_j$ into $k_1,\ldots,k_s$, respectively.  Then there exists $(c_1, \dotsc, c_t) \in C_1 \times \dotsb \times C_t$ such that $f_i(\phi_{1i}(c_1),\dotsc,\phi_{ti}(c_t))\neq 0$ for every $i\in\{1,\ldots,s\}$.
\end{lemma}

\begin{proof}
Since $k[x_1,\dotsc,x_t]=k[x_1,\dotsc,x_{t-1}][x_t]$ for every field~$k$, induction on~$t$ suffices.
\end{proof}

\begin{lemma}\label{lemma:multi-rank-block}
For all $i\in\{1,\ldots,s\}$ and $j\in\{1,\ldots,t\}$, let $A_i$ and $B_i := (B_{i1}\hspace{1mm}\vert\dotsb\vert\hspace{1mm} B_{it})$ be $m_i \times n_i$ matrices with entries in a field $k_i$; suppose that $\rank(B_i) \ge r_i$ and that $B_{ij}$ has size $m_i\times n_{ij}$; and let $C_j$ be a set of size exceeding $\sum_{i=1}^s \min\{r_i,n_{ij}\}$ that can be embedded into $k_1,\ldots,k_s$ via maps $\phi_{j1},\ldots,\phi_{js}$, respectively.  Then there exists $(c_1, \dotsc, c_t)\in C_1 \times \dotsb \times C_t$ such that $\rank(A_i + (\phi_{1i}(c_1)B_{i1}\hspace{1mm}\vert\dotsb\vert\hspace{1mm} \phi_{ti}(c_t)B_{it})) \ge r_i$ for every $i\in\{1,\ldots,s\}$.  An analogous statement holds for the transposes of $A_1,B_1,\ldots,A_s,B_s$.
\end{lemma}

\begin{proof}
For every $i\in\{1,\ldots,s\}$, we may assume, without loss of generality, that $A_i$ and $B_i$ are both $r_i \times r_i$ matrices and that $\det(B_i) \neq 0$. (So, now, $r_i = n_{i1}+ \dotsb + n_{it}$.)  For every $i\in\{1,\ldots,s\}$, let $f_i:=\det(A_i+(x_1B_{i1}\hspace{1mm}\vert\dotsb\vert\hspace{1mm} x_tB_{it})) = \det(A_i) + \dotsb + \det(B_i)(x_1^{n_{i1}}\dotsm x_t^{n_{it}})$ so that $f_i$ is a nonzero polynomial in $k_i[x_1,\dotsc, x_t]$ with every monomial in its support a factor of $x_1^{n_{i1}}\dotsm x_t^{n_{it}}$.  Lemma~\ref{lemma:multi-poly-roots} now finishes the proof of the first claim of the lemma.  The second claim can be verified in a similar manner.
\end{proof}

\begin{lemma}\label{lemma:multi-rank}
For all $i\in\{1,\ldots,s\}$ and $j\in\{1,\ldots,t\}$, let $A_i$ and $B_{ij}$ be $m_i \times n_i$ matrices with entries in a field $k_i$.  Suppose that, for every $i\in\{1,\ldots,s\}$, there exists $j_i\in\{1,\ldots,t\}$ such that $\rank(B_{i,j_i})\geqslant r_i$.  For every $j\in\{1,\ldots,t\}$, let $C_j$ be a set containing at least $1+\sum_{i=1}^s r_i$ elements but at most $\min\{|k_1|,\ldots,|k_s|\}$ elements so that there exist injections $\phi_{j1},\ldots,\phi_{js}$ from $C_j$ to $k_1,\ldots,k_s$, respectively.  Then there exists $(c_1, \dotsc, c_t)\in C_1 \times \dotsb \times C_t$ such that $\rank(A_i + \phi_{1i}(c_1)B_{i1}+\cdots + \phi_{ti}(c_t)B_{it}) \ge r_i$ for every $i\in\{1,\ldots,s\}$.
\end{lemma}

\begin{proof}
For every $i\in\{1,\ldots,s\}$, we reduce to the case in which $m_i=n_i=r_i$ and $\det(B_{i,j_i})\neq 0$ and define $f_i:=\det(A_i + x_1B_{i1}+\cdots + x_tB_{it}) = \det(A_i) + \dotsb + \det(B_{i,j_i})x_{j_i}^{r_i} + \dotsb$, a nonzero polynomial in $k_i[x_1,\dotsc, x_t]$ with every monomial in its support dividing $x_1^{r_i}\dotsm x_t^{r_i}$.  Applying Lemma~\ref{lemma:multi-poly-roots} is all that remains to be done.
\end{proof}

\section{Main results}\label{sec:main}

Our purpose in this section is to achieve several stronger versions of our main theorem.  Our primary message is that neither all of $\Hom_R(N,M)$ nor the entirety of $\Supp_R(N)$ necessarily needs to be studied when estimating $\inj_R(M,N)$ or $\cog_R(N,M)$.  Indeed, it suffices to work with an $R$-submodule of $\Hom_R(N,M)$ admitting large local injective capacities or small local numbers of cogenerators on the finite set $\Ass_R(N)$.

Each observation in this section also contributes something unique to our discussion.  Theorem~\ref{theorem:inj} produces a finite list of maps in $\Hom_R(N,M)$ whose order is linked to local injective capacities on $\Ass_R(N)$, and Theorem~\ref{theorem:cog} serves as its dual for local numbers of cogenerators.  Corollary~\ref{corollary:inj} delineates a local criterion for an $R$-submodule $L$ of~$M$ to intersect trivially with a isomorphic copy of~$N^{\oplus t}$ in~$M$, and Corollary~\ref{corollary:cog} replaces~$L$,~$M$, and $N^{\oplus t}$ in the last statement with~$L^{\oplus t}$,~$M^{\oplus t}$, and~$N$, respectively.  This section ends with Theorem~\ref{theorem:summary}, which summarizes our findings:  A global injective capacity is the infimum of its local analogues on~$\Ass_R(N)$; a global number of cogenerators is the supremum of the corresponding local invariants on~$\Ass_R(N)$; and an $R$-submodule of $\Hom_R(N,M)$ contains an injection if and only if its localizations at the associated primes of $N$ contain injections.

\begin{theorem}\label{theorem:inj}
Let $M$ and $N$ be modules over a commutative ring~$R$ with $N$ nonzero Noetherian, and let $F$ be an $R$-submodule of $\Hom_R(N,M)$.  For every $\p\in\Ass_R(N)$, let $t(\p)$ be a positive integer, and suppose that $t(\p)\leqslant\inj^{F_{\p}}_{R_{\p}}(M_{\p},N_{\p})$.  Next, for every multiplicatively closed subset $S$ of~$R$ avoiding~$\Ann_R(N)$, define $u(S):=\min\{t(\p):\p\in\Ass_R(N),\p\cap S=\varnothing\}$.  Also, let $v:=\max\{t(\p):\p\in\Ass_R(N)\}$.  Then there exist $h_1,\ldots,h_v\in F$ such that the map $S^{-1}(h_1,\ldots,h_{u(S)})$ is injective for every multiplicatively closed subset $S$ of~$R$ avoiding~$\Ann_R(N)$.  Hence, letting $u:=u(\{1\})$, we find that $(h_1,\ldots,h_u)$ is injective.	 
\end{theorem}

To prove this theorem, it suffices to compute $h_1,\ldots,h_v\in F$ such that $(h_1,\ldots,h_{t(\p)})_{\p}$, when restricted to $\Soc_{R_{\p}}(N_{\p}^{\oplus t(\p)})$, is injective for every $\p\in\Ass_R(N)$, as Lemma~\ref{lemma:reduction} indicates below.  To find such maps $h_1,\ldots,h_v$, we use Lemma~\ref{lemma:multi-rank-block} as part of a ``general position" argument somewhat reminiscent of~\cite[Sections~6 and~7]{Bai1}.

\begin{lemma}\label{lemma:reduction}
	Let $M$ and $N$ be modules over a commutative ring~$R$ with $N$ Noetherian.  Let $h\in\Hom_R(N,M)$.  Then $h$ is injective if and only if $h_{\p}$, when restricted to $\Soc_{R_{\p}}(N_{\p})$, is injective for every $\p\in\Ass_R(N)$.
\end{lemma}

\begin{proof}[Proof of Lemma~\ref{lemma:reduction}]
	The forward direction is obvious.  For the reverse direction, let $K:=\ker(h)$.  Then $\Ass_R(K)\subseteq\Ass_R(N)$.  Since $\Soc_{R_{\p}}(K_{\p})=K_{\p}\cap\Soc_{R_{\p}}(N_{\p})=0$ for every $\p\in\Ass_R(N)$, we see that $\Ass_R(K)=\varnothing$.  Hence $K=0$, and so $h$ is injective.
\end{proof}

\begin{proof}[Proof of Theorem~\ref{theorem:inj}]
	The main portion of our proof works only for $\p\in\Ass_R(N)$ such that $|\kappa(\p)|>\dim_{\kappa(\p)}(\Soc_{R_{\p}}(N_{\p}))$.  This condition always holds for the members of $\Ass_R(N)$ not in $\Max(R)$.  To account for the remaining members of $\Ass_R(N)$, we take inspiration from the ``general position" argument in~\cite[Section~6]{Bai1}:  For every $\m\in\Ass_R(N)\cap\Max(R)$, let $s(\m)$ be an element of $R$ that avoids $\m$ but belongs to every other member of $\Ass_R(N)\cap\Max(R)$, and let $d(\m)$ be a member of $F^{\oplus v}\subseteq\Hom_R(N^{\oplus v},M)$ whose first $t(\m)$ components form a map that becomes injective after localizing at $\m$ and whose last $v-t(\m)$ components are zero.  Let
	\[
	(e_1,\ldots,e_v):=\sum_{\m\in\Ass_R(N)\cap\Max(R)} s(\m)d(\m).
	\]
	Then, for every $\m\in\Ass_R(N)\cap\Max(R)$, the map $(e_1,\ldots,e_{t(\m)})_{\m}$ becomes injective when restricted to $\Soc_{R_{\m}}(N_{\m}^{\oplus t(\m)})$.  Hence, if $\Ass_R(N)\subseteq\Max(R)$, then we may take $(h_1,\ldots,h_v):=(e_1,\ldots,e_v)$ and apply Lemma~\ref{lemma:reduction} to close out the proof.
	
	Otherwise, we may continue by revising the ``general position" argument from~\cite[Section~7]{Bai1} in the following way:  Let $\q_1,\ldots,\q_m$ be the distinct members of $\Ass_R(N)\setminus\Max(R)$, and suppose that, for every $\ell\in\{1,\ldots,m\}$, the ideal $\q_1\cap\cdots\cap\q_{\ell-1}$ is not contained in~$\q_{\ell}$.  Fix an $\ell\in\{1,\ldots,m\}$, and choose a map $(f_1,\ldots,f_v)\in F^{\oplus v}$ as follows:  If $\ell=1$, then take $(f_1,\ldots,f_v):=(e_1,\ldots,e_v)$; if $\ell\geqslant 2$, then suppose inductively that, for every $\p\in\Ass_R(N)$ with $\p\not\in\{\q_{\ell},\ldots,\q_m\}$, the map $(f_1,\ldots,f_{t(\p)})_{\p}$ becomes injective when restricted to $\Soc_{R_{\p}}(N_{\p}^{\oplus t(\p)})$.  Let $J$ denote the intersection of the members of $\Ass_R(N)\setminus\{\q_{\ell},\ldots,\q_m\}$, and let $\q:=\q_{\ell}$ and $t:=t(\q)$.  Now let $(g_1,\ldots,g_t)\in F^{\oplus t}$ be a map that becomes injective after localizing at~$\q$.  Since $|(J + \q)/\q| = \infty > \dim_{\kappa(\q)}(\Soc_{R_{\q}}(N_\q))$, we may use Lemma~\ref{lemma:multi-rank-block} to conclude that there exist $c_1,\ldots,c_t \in J$ such that $(f_1, \dotsc, f_t) + (c_1g_1, \dotsc, c_tg_t)$ becomes injective after localizing at $\q$ and restricting to $\Soc_{R_{\q}}(N_{\q}^{\oplus t})$.  We may thus take $(f_1+c_1g_1, \dotsc, f_t+c_tg_t,f_{t+1},\ldots,f_v)$ to complete our inductive step, and we can conclude the proof with an application of Lemma~\ref{lemma:reduction}.
\end{proof}

\begin{corollary}\label{corollary:inj}
	Let $A \le C$ and $B \le D$ be modules over a commutative ring~$R$ with $A$ nonzero Noetherian, and let $G$ be an $R$-submodule of $\Hom_R(C,D)$.  For every $\p\in\Ass_R(A)$, let $t(\p)$ be a positive integer, and suppose that there exists a map in $G_{\p}^{\oplus t(\p)}\subseteq\Hom_{R_{\p}}(C_{\p}^{\oplus t(\p)},D_{\p})$ under which the preimage of $B_{\p}$ has trivial intersection with $A_{\p}^{\oplus t(\p)}$. Next, for every multiplicatively closed subset $S$ of~$R$ avoiding~$\Ann_R(A)$, define $u(S):=\min\{t(\p):\p\in\Ass_R(A),\p\cap S=\varnothing\}$.  Also, let $v:=\max\{t(\p):\p\in\Ass_R(A)\}$. Then there exist $h_1,\ldots,h_v\in G$ such that the preimage of $S^{-1}B$ under $S^{-1}(h_1,\ldots,h_{u(S)})$ has trivial intersection with $S^{-1}A^{\oplus u(S)}$ for every multiplicatively closed subset $S$ of~$R$ avoiding~$\Ann_R(A)$.  Hence, if $u:=u(\{1\})$, then the preimage of $B$ under $(h_1,\ldots,h_u)$ has trivial intersection with $A^{\oplus u}$. 
\end{corollary}

\begin{proof}
	This is an application of Theorem~\ref{theorem:inj} to $\inj_R^F(M,N)$ with $N=A$, $M = D/B$, and $F$ being the image of $G$ under the natural map $\Hom_R(C,D) \to \Hom_R(A,D/B)$.
\end{proof}

\begin{theorem}\label{theorem:cog}
	Let $M$ and $N$ be modules over a commutative ring~$R$ with $N$ nonzero Noetherian, and let $F$ be an $R$-submodule of $\Hom_R(N,M)$.  For every $\p\in\Ass_R(N)$, suppose that $\cog^{F_{\p}}_{R_{\p}}(N_{\p},M_{\p})$ is finite, and let $t(\p)$ be an integer $\geqslant\cog^{F_{\p}}_{R_{\p}}(N_{\p},M_{\p})$.  For every multiplicatively closed subset $S$ of~$R$ avoiding~$\Ann_R(N)$, let $v(S):=\max\{t(\p):\p\in\Ass_R(N),\hspace{1mm}\p\cap S=\varnothing\}$.  Also, let $v:=v(\{1\})$.  Then there exist $h_1,\ldots,h_v\in F$ such that $S^{-1}(h_1,\ldots,h_{v(S)})^{\top}$ is injective for every multiplicatively closed subset $S$ of~$R$ avoiding~$\Ann_R(N)$.  Hence $(h_1,\ldots,h_v)^{\top}$ is injective.	
\end{theorem}

\begin{proof}
	We may proceed as in the proof of Theorem~\ref{theorem:inj} except that we must take the transpose of every matrix involved and consider $N_{\p}$ and $M_{\p}^{\oplus t(\p)}$ instead of $N_{\p}^{\oplus t(\p)}$ and $M_{\p}$ for every $\p\in\Ass_R(N)$.  Note that this proof uses the second claim of Lemma~\ref{lemma:multi-rank-block} rather than the first.
\end{proof}

\begin{corollary}\label{corollary:cog}
	Let $A \le C$ and $B \le D$ be modules over a commutative ring~$R$ with $A$ nonzero Noetherian, and let $G$ be an $R$-submodule of $\Hom_R(C,D)$. For every $\p\in\Ass_R(A)$, let $t(\p)$ be an integer, and suppose that there is a map in $G_{\p}^{\oplus t(\p)}\subseteq\Hom_{R_{\p}}(C_{\p},D_{\p}^{\oplus t(\p)})$ under which the preimage of $B_{\p}^{\oplus t(\p)}$ has trivial intersection with $A_{\p}$. Next, for every multiplicatively closed subset $S$ of~$R$ avoiding~$\Ann_R(A)$, define $v(S):=\max\{t(\p):\p\in\Ass_R(A),\p\cap S=\varnothing\}$.  Also, let $v:=v(\{1\})$. Then there exist $h_1,\ldots,h_v\in G$ such that the preimage of $S^{-1}B^{\oplus v(S)}$ under $S^{-1}(h_1,\ldots,h_{v(S)})^{\top}$ has trivial intersection with $S^{-1}A$ for every multiplicatively closed subset $S$ of~$R$ avoiding~$\Ann_R(A)$.  Hence, if $v:=v(\{1\})$, then the preimage of $B^{\oplus v}$ under $(h_1,\ldots,h_v)^{\top}$ has trivial intersection with $A$. 	
\end{corollary}

\begin{proof}
	This is an application of Theorem~\ref{theorem:cog} to $\cog_R^F(N,M)$ with $N=A$, $M = D/B$, and $F$ being the image of $G$ under the natural map $\Hom_R(C,D) \to \Hom_R(A,D/B)$.
\end{proof}

\begin{theorem}\label{theorem:summary}
	Let $M$ and $N$ be modules over a commutative ring~$R$ with $N$ Noetherian, and let $F$ be an $R$-submodule of $\Hom_R(N,M)$.  Then the following statements hold:
	\begin{enumerate}
		\item $\inj^F_R(M,N)=\inf\{\inj^{F_{\p}}_{R_{\p}}(M_{\p},N_{\p}):\p\in\Ass_R(N)\}$.
		\item $\cog^F_R(N,M)=\sup\{\cog^{F_{\p}}_{R_{\p}}(N_{\p},M_{\p}):\p\in\Ass_R(N)\}$.		
		\item $F$ contains an injection if and only if $F_{\p}$ contains an injection for every $\p\in\Ass_R(N)$.
	\end{enumerate}
In fact, in each of the statements above, we can replace $\Ass_R(N)$ with the set consisting of just the maximal members of~$\Ass_R(N)$.
\end{theorem}

\begin{proof}
	(1)  Let $t$ and $u$ denote the left and right sides of the asserted equation, respectively.  Since an injective map remains injective upon localization, $t\leqslant u$.  To prove the reverse inequality, we note that either $u=0$, in which case $u\leqslant t$ automatically, or else Theorem~\ref{theorem:inj} shows that $u\leqslant t$, including the case in which $u = \infty$.
	
	(2)  Label the left and right sides of the proposed equality as $t$ and~$v$, respectively.  If $v$ is infinite, then $t\leqslant v$ at once; otherwise, Theorem~\ref{theorem:cog} suffices.  We may also consider two cases to prove that $v\leqslant t$:  If $t$ is infinite, then $v\leqslant t$ immediately; otherwise, we may appeal to the exactness of localization.
	
	(3)  We may apply either Part~(1) or Part~(2) of the present theorem since $F$ contains an injection if and only if $\inj^F_R(M,N)\geqslant 1$ if and only if $\cog^F_R(N,M)\leqslant 1$.
	
	Having proved Parts (1), (2), and (3), we yield the last claim of the theorem by recalling once again that localization preserves injectivity.
\end{proof}

	\section{A graded embedding theorem}\label{sec:inj-graded}

	We close this paper with some thoughts on the existence of a homogeneous monomorphism between two $\ZZ$-graded modules.  We refer the reader to~\cite[Section~1.5]{BH} for an introduction to $\ZZ$-graded rings and modules but summarize some standard notation and results here for the reader's convenience.  To begin with, a \textit{$\ZZ$-grading} on a commutative ring $R$ is a decomposition of $R$ into a direct sum of $\ZZ$-modules $\ldots,R_{-1},R_0,R_1,\ldots$ such that, for all $i,j\in\ZZ$ and $a\in R_i$ and $b\in R_j$, the element $ab$ belongs to $R_{i+j}$.  In this section, $R$ denotes a commutative ring with a fixed $\ZZ$-grading.  A \textit{$\ZZ$-grading} on an $R$-module $M$ is a decomposition of $M$ into a direct sum of $\ZZ$-modules $\ldots,M_{-1},M_0,M_1,\ldots$ such that, for all $i,j\in\ZZ$ and $a\in R_i$ and $x\in M_j$, the element $ax$ belongs to $M_{i+j}$.  Note that $R_0$ is a $\ZZ$-graded commutative ring and that every $\ZZ$-graded $R$-module is naturally a $\ZZ$-graded $R_0$-module.  Henceforth, $M$ and $N$ signify $R$-modules with fixed $\ZZ$-gradings.  For every $i\in\ZZ$, a member of $M_i$ is called a \textit{homogeneous element of $M$ of degree~$i$}, and $M_i$ is referred to as the \textit{$i$th homogeneous component of~$M$}.  Also, for every~$i\in\ZZ$, the symbol $M[i]$ stands for the $\ZZ$-graded $R$-module such that, for every~$j\in\ZZ$, we have $(M[i])_j=M_{i+j}$.	If $L$ is a $\ZZ$-graded $R$-submodule of~$M$, then $M/L$ is naturally a $\ZZ$-graded $R$-module such that, for every $i\in\ZZ$, we have $(M/L)_i\cong M_i/L_i$ as $R_0$-modules.  We assume that $N$ is Noetherian, which is equivalent to assuming that every ascending chain of $\ZZ$-graded $R$-submodules of $N$ stabilizes.  As a consequence, every member of $\Ass_R(N)$ is a $\ZZ$-graded ideal of $R$, and $\Hom_R(N,M)$ is a $\ZZ$-graded $R$-module such that, for every $i\in\ZZ$, the $i$th homogeneous component of $\Hom_R(N,M)$ is the set of all $f$ such that, for every~$j\in\ZZ$, we have $f(N_j)\subseteq M_{i+j}$.  For every $\p\in\Spec(R)$, if $S$ is the set of all homogeneous elements of $R$ avoiding~$\p$, then the \textit{homogeneous localization $M_{(\p)}$ of $M$ at $\p$} refers to the $\ZZ$-graded $S^{-1}R$-module~$S^{-1}M$ such that, for every~$i\in\ZZ$, the $i$th homogeneous component of $S^{-1}M$ is the set of all $x/s$ such that $x\in M_{i+j}$ and $s\in R_j\setminus\p$ for some~$j\in\ZZ$.  For every $\ZZ$-graded $\p\in\Spec(R)$, the \textit{socle of $N_{(\p)}$ over $R_{(\p)}$}, denoted $\Soc_{R_{(\p)}}(N_{(\p)})$, is the $\ZZ$-graded module over $R_{(\p)}/\p_{(\p)}$ consisting of all elements of $N_{(\p)}$ annihilated by~$\p_{(\p)}$.  A prime $\p$ of $R$ is associated to $N$ if and only if $\p$ is $\ZZ$-graded and $\Soc_{R_{(\p)}}(N_{(\p)})\neq 0$.     
	
	We now offer a $\ZZ$-graded version of Theorem~\ref{theorem:summary} in the case that $R_0/\p_0$ is sufficiently large for every $\p\in\Ass_R(N)$.  
	
	\begin{theorem}\label{theorem:inj-graded}
		Let $M$ and $N$ be $\ZZ$-graded modules over a $\ZZ$-graded commutative ring~$R$ with $N$ Noetherian.  For every $\p\in\Ass_R(N)$, let $r(\p):=\rank_{R_{(\p)}/\p_{(\p)}}(\Soc_{R_{(\p)}}(N_{(\p)}))$; let $P(\p):=\{\p'\in\Ass_R(N):\p'_0=\p_0\}$; and assume that $|R_0/\p_0|>\sum_{\p'\in P(\p)} r(\p')$.  Let $F$ be a $\ZZ$-graded $R$-submodule of $\Hom_R(N,M)$.  Assume that there exists a degree $i$ such that, for every $\p \in \Ass_R(N)$, the $i$th homogeneous component $F_i$ of $F$ contains a map that becomes injective after homogeneous localization at~$\p$.  Then $F_i$ contains an injection.
	\end{theorem}

	The following example shows that we cannot remove the uniformity condition on degree in Theorem~\ref{theorem:inj-graded}:  Suppose that $R=R_0$ and that $R$ has two incomparable prime ideals $\p$ and~$\q$.  Let $N: = (R/\p) \oplus (R/\q)$, and let $M := (R/\p)[-1] \oplus (R/\q)$.  Then $\Hom_R(N,M)$ contains homogeneous maps $f$ and $g$ such that $f_{(\p)}$ and $g_{(\q)}$ are injective, but $\Hom_R(N,M)$ does not itself contain a homogeneous injection.

	Since the uniformity condition in Theorem~\ref{theorem:inj-graded} cannot be removed, we would like to mention one way to guarantee that this condition is satisfied:  First assume that, for every $\p\in\Ass_R(N)$, there exists a homogeneous map $f(\p)$ in $\Hom_R(N,M)$ of degree $i(\p)$ that becomes injective after homogeneous localization at~$\p$.  Next, suppose that there exists $i \in \ZZ$ such that, for every $\p\in\Ass_R(N)$, we have $R_{i-i(\p)}\not\subseteq \p$.  Then, for every $\p\in\Ass_R(N)$, we can choose $s(\p) \in R_{i - i(\p)}\setminus \p$ and then replace $f(\p)$ with $s(\p)f(\p) \in F_i$.

	We can reduce the proof of Theorem~\ref{theorem:inj-graded} to finding a degree-$i$ homogeneous map in $\Hom_R(N,M)$ that becomes injective after homogeneously localizing at any associated prime of~$N$ and restricting to the socle, as the following lemma indicates.  We omit the proof of the following lemma since it is similar to the one for Lemma~\ref{lemma:reduction}.
	
	\begin{lemma}\label{lemma:reduction-graded}
		Let $M$ and $N$ be $\ZZ$-graded modules over a $\ZZ$-graded commutative ring~$R$ with $N$ Noetherian.  Let $h$ be a homogeneous map in $\Hom_R(N,M)$.  Then $h$ is injective if and only if $h_{(\p)}$, when restricted to $\Soc_{R_{(\p)}}(N_{(\p)})$, is injective for every $\p\in\Ass_R(N)$.
	\end{lemma}

	\begin{proof}[Proof of Theorem~\ref{theorem:inj-graded}]
	List the members of the set $Q:=\{\p_0 : \p \in \Ass_R(N)\}$ so that no member contains any of its predecessors.  Let $\q\in Q$, and suppose inductively that there exists $f\in F_i$ such that, for every $\p\in\Ass_R(N)$ with $\p_0$ a predecessor of $\q$, the map $f_{(\p)}$ becomes injective when restricted to $\Soc_{R_{(\p)}}(N_{(\p)})$.  Let $J$ denote the intersection of the predecessors of~$\q$, and let $P$ denote the fiber of $\q$ in $\Ass_R(N)$.  For every $\p\in P$, let $g(\p)$ be a map in $F_i$ that becomes injective after homogeneously localizing at~$\p$.  Note that $|(J+\q)/\q|>\sum_{\p\in P} r(\p)$.  Hence, by Lemma~\ref{lemma:multi-rank}, there exists a function $c:P\rightarrow J$ such that, for every $\p\in P$, the map $f + \sum_{\p'\in P} c(\p')g(\p')$ becomes injective after homogeneously localizing at $\p$ and restricting to $\Soc_{R_{(\p)}}(N_{(\p)})$.  Induction, followed by an application of Lemma~\ref{lemma:reduction-graded}, brings the proof to a close.
	\end{proof}

	\section*{Acknowledgements}\label{sec:acknowledgements}

The first author would like to thank Ela and Olgur Celikbas, Guantao Chen, Florian Enescu, Neil Epstein, Karl Schwede, Alexandra Smirnova, and Janet Vassilev for providing opportunities to speak about the results of this paper.  The first author would also like to thank Craig Huneke for offering helpful feedback on the content of the paper.  This research did not receive any specific grant from funding agencies in the public, commercial, or not-for-profit sectors.

\bibliographystyle{amsplain}
\bibliography{baidya-references}

\providecommand{\bysame}{\leavevmode\hbox to3em{\hrulefill}\thinspace}
\providecommand{\MR}{\relax\ifhmode\unskip\space\fi MR }
\providecommand{\MRhref}[2]{%
  \href{http://www.ams.org/mathscinet-getitem?mr=#1}{#2}
}
\providecommand{\href}[2]{#2}
\begin{thebibliography}{10}

\bibitem{Bai1}
R.~Baidya, \emph{Surjective capacity and splitting capacity}, J.~Algebra
  \textbf{537} (2019), 343--380.

\bibitem{Bai4}
\bysame, \emph{The stable rank and number of generators of one module with
  respect to another}, In progress.

\bibitem{Bas}
H.~Bass, \emph{{$K$}-theory and stable algebra}, Publ.~Math.~Inst.~Hautes
  {\'E}tudes Sci. \textbf{22} (1964), 5--60.

\bibitem{BH}
W.~Bruns and J.~Herzog, \emph{{C}ohen--{M}acaulay rings}, Cambridge University
  Press, 1993.

\bibitem{CLQ}
T.~Coquand, H.~Lombardi, and C.~Quitt\'e, \emph{Generating non-{N}oetherian
  modules constructively}, Manuscripta Math. \textbf{115} (2004), 513--520.

\bibitem{Cou2}
S.~C. Coutinho, \emph{Basic element theory of noncommutative {N}oetherian
  rings}, J.~Algebra \textbf{144} (1991), 24--42.

\bibitem{Cou}
\bysame, \emph{Generating modules efficiently over noncommutative {N}oetherian
  rings}, Trans.~Amer.~Math.~Soc. \textbf{323} (1991), 843--856.

\bibitem{DSPY}
A.~De~Stefani, T.~Polstra, and Y.~Yao, \emph{Generalizing {S}erre's {S}plitting
  {T}heorem and {B}ass's {C}ancellation {T}heorem via free-basic elements},
  Proc.~Amer.~Math.~Soc. \textbf{146} (2018), 1417--1430.

\bibitem{EE}
D.~Eisenbud and E.~G. Evans, Jr., \emph{Generating modules efficiently:
  theorems from algebraic {$K$}-theory}, J.~Algebra \textbf{27} (1973),
  278--305.

\bibitem{For}
O.~Forster, \emph{\"{U}ber die {A}nzahl der {E}rzeugenden eines {I}deals in
  einem {N}oetherschen {R}ing}, Math.~Z. \textbf{84} (1964), 80--87.

\bibitem{Heit}
R.~C. Heitmann, \emph{Generating non-{N}oetherian modules efficiently},
  Michigan Math.~J. \textbf{31} (1984), 167--180.

\bibitem{Lyu}
G.~Lyubeznik, \emph{The number of generators of modules over polynomial rings},
  Proc.~Amer.~Math.~Soc. \textbf{103} (1988), 1037--1040.

\bibitem{Ser}
J.-P. Serre, \emph{Modules projectifs et espaces fibr\'es \`a fibre
  vectorielle}, S\'eminaire Dubreil. Alg\`ebre et th\'eorie des nombres
  \textbf{11} (1957/1958), no.~2, 1--18.

\bibitem{Sta}
J.~T. Stafford, \emph{Generating modules efficiently: algebraic {$K$}-theory
  for noncommutative {N}oetherian rings}, J.~Algebra \textbf{69} (1981), no.~2,
  312--346.

\bibitem{Swa2}
R.~G. Swan, \emph{Vector bundles and projective modules},
  Trans.~Amer.~Math.~Soc. \textbf{105} (1962), 264--277.

\bibitem{Swa}
\bysame, \emph{The number of generators of a module}, Math.~Z. \textbf{102}
  (1967), 318--322.

\bibitem{War2}
R.~B. Warfield, Jr., \emph{The number of generators of a module over a fully
  bounded ring}, J.~Algebra \textbf{66} (1980), 425--447.

\end{thebibliography}

\end{document}